\documentclass{siamltex}
\usepackage{amsfonts}
\usepackage{sw20siam}
\usepackage{amsmath}
\usepackage{amssymb}
\usepackage{mathrsfs}
\usepackage{graphicx}
\usepackage{mathtools}

\newtheorem{hypothesis}[theorem]{Assumption}
\newtheorem{problem}[theorem]{Problem}
\newtheorem{example}[theorem]{\textit{Example}}
\newtheorem{remark}[theorem]{\textit{Remark}}

\newcommand{\eps}{\varepsilon}  
\newcommand{\vf}{\varphi}
\newcommand{\al}{\alpha}

\newcommand{\de}{\delta}
\newcommand{\la}{\lambda}
\newcommand{\si}{\sigma}

\newcommand{\om}{\omega}

\newcommand{\md}{\mathrm{d}}   
\newcommand{\vd}{\,\md}
\newcommand{\me}{\mathrm{e}}   
\newcommand{\ch}{\chi}   

\newcommand{\essinf}{\mathop{\mathrm{ess}\inf}}
\newcommand{\minev}{\lambda_{*}}   
\newcommand{\tr}{\top}   
\newcommand{\Tr}{\mathrm{Tr}}  
\newcommand{\p}{p}  

\newcommand{\con}{\kappa}  

\newcommand{\R}{\mathbb{R}}   
\newcommand{\Sym}{\mathbb{S}}  

\newcommand{\PS}{\Omega}  
\newcommand{\E}{\mathbb{E}}  
\newcommand{\Prob}{\mathbb{P}}  
\newcommand{\Filt}{\mathscr{F}}  
\newcommand{\filt}{\mathbb{F}}  
\newcommand{\Pred}{\mathscr{P}}  
\newcommand{\BM}{w}  
\newcommand{\st}{\tau}  

\newcommand{\isum}{\sum\nolimits_i}  
\newcommand{\tsum}{\sum\nolimits}

\newcommand{\svbset}{\mathfrak{S}}   
\newcommand{\Rc}{P}   
\newcommand{\Df}{\varLambda}   
\newcommand{\Sw}{Q}   
\newcommand{\Cw}{R}   
\newcommand{\Tw}{H}   

\newcommand{\Ctr}{\varGamma}
\newcommand{\Ch}{\varDelta}
\newcommand{\func}{\Theta}

\newcommand{\sRc}{\bar{\Rc}}   
\newcommand{\sSw}{\bar{\Sw}}   
\newcommand{\sCw}{\bar{\Cw}}   
\newcommand{\sTw}{\bar{\Tw}}   

\newcommand{\iRc}{F}   

\newcommand{\dr}{\mathring}  
\newcommand{\df}{\breve}  

\newcommand{\bd}{\mathrm{b}}

\newcommand{\cpr}{\mathcal{S}}   
\newcommand{\ipr}{\mathcal{H}}   

\newcommand{\itv}[1]{[\![ #1 )\mspace{-5mu})}
\newcommand{\citv}[1]{[\![ #1 ]\!]}
\newcommand{\oitv}[1]{(\mspace{-5mu}( #1 ]\!]}

\newcommand{\vA}{\tilde{A}}

\newcommand{\vC}{\tilde{C}}

\newcommand{\vX}{\tilde{X}}
\newcommand{\vvX}{\varPsi}

\title{Solvability conditions for indefinite linear quadratic optimal stochastic control problems and associated stochastic Riccati equations}

\author{Kai~Du\thanks{Institute for Mathematics and Its Applications, School of Mathematics and Applied Statistics,
University of Wollongong,
Wollongong, NSW 2522, Australia ({\tt kaid@uow.edu.au}).}
}  

\begin{document}
\maketitle

\begin{abstract}
A linear quadratic optimal stochastic control problem with random coefficients and indefinite state/control weight costs is usually linked to an indefinite stochastic Riccati equation (SRE) which is a matrix-valued quadratic backward stochastic differential equation along with an algebraic constraint involving the unknown.
Either the optimal control problem or the SRE is solvable only if the given data satisfy a certain structure condition that has yet to be precisely defined.
In this paper, by introducing a notion of subsolution for the SRE, we derive several novel sufficient conditions for the existence and uniqueness of the solution to the SRE and for the solvability of the associated optimal stochastic control problem.
\end{abstract}
\keyphrases{linear quadratic optimal stochastic control, stochastic Riccati equation, backward stochastic differential equation, subsolution, solvability condition}
\AMclass{60H10; 49N10, 93E20}

\section{Introduction}

A classical form of a linear quadratic optimal stochastic control (SLQ for short) problem is to minimize the quadratic cost functional
\begin{equation}\label{cost}
J(u;\xi) = \E\,\biggl\{x(T)^\tr \Tw x(T) + \int_0^T \bigl[ u^\tr(t) \Cw(t) u(t) + x^\tr \Sw(t) x(t)\bigr]\vd t \biggr\}
\end{equation}
with the control $u = u(\cdot)$ being a square-integrable adapted process
and the state $x = x(\cdot)$ being the solution to the linear stochastic control system 
\begin{equation}\label{1-004b}
\md x = (Ax + Bu)\,\md t + \isum (C_i x + D_i u)\,\md \BM^i_t,\quad
x(0) = \xi \in \R^n,
\end{equation}
where $T$ is a given final time, $\BM$ is a $d$-dimensional Wiener process, $A,B,C,D,\Cw,\Sw,$ and $\Tw$ are given coefficients, in particular, $\Cw,\Sw,$ and $\Tw$ are all symmetric matrix-valued processes,
and where we have used a convenient notation
\[
\isum := \tsum_{i=1}^{d}
\]
that will also be used throughout the paper.
As in \eqref{1-004b}, the time variable $t$ will be suppressed for simplicity in many circumstances, when no confusion occurs.
We assume in this article \emph{all the given coefficients to be random}.

Under a {definiteness assumption} that $\Sw$ and $\Tw$ are positive semi-definite and $\Cw$ is positive definite, Bismut \cite{bismut1976linear} made a deep investigation into the above control problem. 
To characterize the minimal cost and construct the optimal feedback control, he formally derived a backward stochastic differential equation (BSDE) called the stochastic Riccati equation (SRE) as follows: 
\begin{subequations}
{\small{\begin{equation}\label{1-01a}
\begin{aligned}
& \md \Rc  =  \tsum_i \Df_{i}\,\md \BM^{i}_t - \Bigl[ A^\tr \Rc  + \Rc  A + \tsum_i (C_{i}^\tr \Rc  C_{i} + C_{i}^\tr \Df_{i} + \Df_{i} C_{i}) + \Sw \Bigr]\md t \\
&\qquad + \Bigl[\Rc B + \tsum_i (C_{i}^\tr \Rc + \Df_{i}) D_{i}\Bigr] \Bigl(\Cw  +  \tsum_i D_{i}^\tr \Rc  D_{i} \Bigr)^{\!-1} \Bigl[B^\tr \Rc  + \tsum_i D_{i}^\tr( \Rc C_{i} + \Df_{i}) \Bigr]\md t,\\
& \Rc(T) = \Tw,
\end{aligned}
\end{equation}}}where the unknown is the matrix-valued process $(\Rc,\Df_1,\dots,\Df_d)$ adapted to the filtration generated by $\BM$.
The minimum of $J(\,\cdot\,;\xi)$ coincides with $\xi^\tr \Rc(0) \xi$ once the SRE is solvable ``properly''.
However, the existence and uniqueness of the solution of \eqref{1-01a} was not completely proved in his work, although he had showed that the original control problem has a unique solution.
He left the solvability of the SRE as an open problem which was resolved decades later by Tang~\cite{tang2003general}\footnote{Recently, Tang~\cite{tang2014dynamic} gave another approach to this problem via the dynamic programming principle.}. 

The systematic study on SLQ problems without the definiteness assumption was initiated by Chen et al.~\cite{chen1998stochastic} who observed that an SLQ problem where $R$ is possibly indefinite may still be solvable.
This finding has triggered an extensive research on the so-called indefinite SLQ problem that has applications in many practical areas, especially in finance (see~\cite{zhou2000continuous,kohlmann2003multidimensional,kohlmann2003minimization,yu2013equivalent} for example).
In~\cite{chen1998stochastic} they also formulated a related indefinite SRE combining \eqref{1-01a} with the constraint
\begin{equation}\label{1-02a}
\Cw  +  \tsum_i D_{i}^\tr \Rc  D_{i} > 0
\quad \text{over } [0,T],
\end{equation}
\label{1-01}\end{subequations}
and proved that the solvability of this equation yields the well-posedness of the original control problem.
This key fact catalyzed quite a few works investigating the existence and uniqueness of the solution to the indefinite SRE.
As indicated in the existing literature, the solvability of \eqref{1-01} is by no means unconditional (see~\cite{chen1998stochastic} for ill-posed examples);
in other words, the equation may have no solution if $\Cw,\Sw$ or $\Tw$ is ``too negative''.
The problem is then to specify the conditions that the given data must satisfy to ensure the solvability of indefinite SREs.
As far as we know, the existing results are limited to several very special cases (see~\cite{hu2003indefinite, qian2013existence}).

In this paper we derive several novel sufficient conditions that ensure the existence and uniqueness of the solution to the indefinite SRE and also imply the solvability of the associated SLQ problem.
According to our understanding,
the constraint \eqref{1-02a} seems to be some kind of coercivity condition that plays a similar role to what the positive definiteness of $\Cw$ does in the definite case. 
But it is too implicit to use.
Our idea is to reveal the coercivity to some degree by means of a new-defined notion of ``subsolution'' for SREs (see Definition~\ref{def2-03} below).
We prove that the existence of subsolutions of \eqref{1-01} implies the well-posedness of the related SLQ problem;
moreover, if SRE~\eqref{1-01} has a subsolution in a strict sense (see Theorem~\ref{th3-01} below), then the equation is solvable and the associated SLQ problem admits a unique optimal feedback control.
The original problem is largely converted into finding the new object of the equation.
A subsolution is an adapted process that satisfies only an inequality form of \eqref{1-01a} --- this relaxing gives us more probabilities to find the target.
Further, considering subsolutions of certain particular forms will bring us several practicable criteria of solvable SREs.
Consequently, we recover many existing results on the solvability of \eqref{1-01}, for instance, obtained in \cite{hu2003indefinite, tang2003general, qian2013existence}.
The proof of Theorem~\ref{th3-01} below occupies most of the technical part of our argument,
in which we borrow an idea from Tang~\cite{tang2003general},
that is, in a nutshell, as long as an associated forward-backward SDE is solvable, a solution of the SRE can be constructed by using the solution of the former ---
we succeed to verify the precondition under our setting, and then achieve our aim.

The rest of the paper is organized as follows.
Section~2 gives a precise formulation of the problems by introducing several notation and definitions.
Section~3 is mainly devoted to the statement of our main results, including some remarks and examples.
Section~4 is the most technical part, containing the proofs of several auxiliary lemmas and the main results.

\section{Preliminaries}

Let $(\PS,\Filt,\filt,\Prob)$ be a
filtered probability space where the filtration $\filt = (\Filt_{t})_{t\geq 0}$ is generated by a $d$-dimensional standard Wiener process $\BM=\{\BM_t;t\geq 0\}$ and satisfies the usual conditions,
$\Pred$ be the predictable $\sigma$-algebra associated with $\filt$.
Fix a finite terminal time $T$. 

Let $\R^n$ be the $d$-dimensional Euclidean space, and $\R^{n\times m}$ the set of $n\times m$ matrices.
We identify $\R^n$ and $\R^{n \times 1}$, 
and use $|M| = [\mathrm{Tr}(M^\tr M)]^{1/2}$ as the norm of $\R^{n\times m}$.
Denote by $\Sym^n$ the set of symmetric $n\times n$ matrices. 
The inequality signs are used to express the usual semi-order of symmetric matrices.
For $\Sym^n$-valued functions (including processes) $M$ and $N$, the expression $M \gg N$ means that $M-N$ is uniformly positive definite almost everywhere (a.e.), i.e., $M-N\ge \de I_n$ a.e. for some $\de > 0$;
the meaning of ``$\ll$'' is obvious.

For stopping times $\si $ and $\tau$ such that $\si\le \tau$, we define
\[
\itv{\si,\tau} = \{(t,\om): t\in [\si(\om),\tau(\om))\};
\]
similarly, we will also use $\citv{\si,\tau}$ and $\oitv{\si,\tau}$.
For $p \in [1,\infty]$, we write
\begin{align*}
\ipr^\p(\si,\tau;\R^{n\times m}) :=\;& L^\p (\itv{\si,\tau},\Pred,\R^{n\times m}),\\
\cpr^\p(\si,\tau;\R^{n\times m}) :=\;& \ipr^\p(\R^{n\times m}) \cap L^\p(\PS;C([\si,\tau];\R^{n\times m})),
\end{align*}
and simply
$$\ipr^\p(\R^{n\times m}) = \ipr^\p(0,T;\R^{n\times m}),
\quad 
\cpr^\p(\R^{n\times m}) = \cpr^\p(0,T;\R^{n\times m}).$$

\begin{definition}\label{def2-03}\rm
$\svbset$ denotes the set of all $\Sym^n$-valued continuous processes $V(\cdot)$ such that
\begin{equation*}
\md V(t) = \dr{V}(t)\vd t + \isum \df{V}_i(t)\vd \BM^i_t
\quad\text{with }\
(\dr{V}, \df{V}_i)\in \ipr^1\times\ipr^2(\Sym^n);
\end{equation*}
Elements of this set are defined up to indistinguishability.
$\svbset^{\bd}$ consists of all bounded processes in $\svbset$.
We write $\df{V} = (\df{V}_1,\dots,\df{V}_d)$.
\end{definition}

With these preparations, let us restate the main problems. 
The following assumption is in force \emph{throughout the paper}.

\begin{hypothesis}\rm
The data $(A,B,C,D;\Cw,\Sw,\Tw)$ satisfy that
\begin{equation*}
\bigg\{\begin{aligned}
& A,C_{i}\in \ipr^\infty(\R^{n \times n}),
\quad
B,D_{i} \in \ipr^\infty(\R^{n \times k}),
\quad i=1,\dots, d.\\
& \Cw \in \ipr^\infty(\Sym^{k}),\quad
\Sw\in \ipr^\infty(\Sym^{n}),
\quad
\Tw \in L^\infty(\PS,\Filt_T, \Sym^{n}).
\end{aligned}
\end{equation*}
\end{hypothesis}

\begin{problem}[LQ optimal stochastic control]\label{prob-a}\rm
Minimize the cost functional \eqref{cost} over $u\in \ipr^2(\R^k)$ 
subject to the control system \eqref{1-004b}.
Define the value function
\begin{equation*}
V(\xi) = \inf_{u\in \ipr^2(\R^{k})} J(u;\xi).
\end{equation*}
The problem is said to be \emph{well-posed} if $V(0) > -\infty$, 
to be \emph{solvable} if for each $\xi\in\R^n$ there is a control $u^* \in \ipr^2(\R^k)$ depending on $\xi$ such that $V(\xi) = J(u^*;\xi)$.
We will refer this optimal control problem as SLQ~$(A,B,C,D;\Cw,\Sw,\Tw)$, or simply, SLQ~$(\Cw,\Sw,\Tw)$ in some circumstances.
\end{problem}

\begin{problem}[stochastic Riccati equation]\label{prob-b}\rm
Define the following functions associated with the parameters $(A,B,C,D;\Cw,\Sw)$: 
\begin{subequations}
\begin{equation}\label{2-03}
\begin{aligned}
\Ch(\Rc) :=\,& \Cw  +  \tsum_i D_{i}^\tr \Rc  D_{i},\\ 
\Ctr(\Rc,\Df) : =\,& - \Ch(\Rc)^{-1} \Bigl[B^\tr\Rc + \tsum_i D_{i}^\tr (\Rc C_{i} + \Df_{i})\Bigr],\\
\func(\Rc,\Df)
:= \, &  A^\tr \Rc  + \Rc  A + \tsum_i (C_{i}^\tr \Rc  C_{i} + C_{i}^\tr \Df_{i} + \Df_{i} C_{i}) + \Sw\\
& - \Ctr(\Rc,\Df)^\tr \Ch(\Rc)\Ctr(\Rc,\Df).
\end{aligned}
\end{equation}
The problem is to find a $\Rc\in \svbset$ such that
\begin{equation}\label{2-01b}
\begin{aligned}
\dr{\Rc} + \func(\Rc,\df{\Rc})  = 0,\quad
\Ch(\Rc)  > 0,\quad
\Rc(T)  = \Tw.
\end{aligned}
\end{equation}
\label{2-01}\end{subequations}
A solution $\Rc$ is said to be bounded if $\Rc\in\svbset^\bd$.
Here and in what follows, the notation $\dr{\Rc}$ and $\df{\Rc}$ are understood in the sense of Definition~\ref{def2-03}. We will refer \eqref{2-01} as SRE~$(A,B,C,D;\Cw,\Sw,\Tw)$, or simply, SRE~$(\Cw,\Sw,\Tw)$ in some circumstances.
\end{problem}

In order to define our sufficient solvability conditions for SLQs and SREs,
we propose an auxiliary notion as follows.

\begin{definition}\label{def2-02}\rm
$\iRc \in \svbset$ is called a \emph{subsolution} to SRE~$(\Cw, \Sw ,\Tw)$ if 
\begin{equation}\label{2-05}
\begin{aligned}
\dr{\iRc} + \func(\iRc,\df{\iRc}) \ge 0,\quad
\Ch(\iRc)  > 0,\quad
\iRc(T)  \le \Tw.
\end{aligned}
\end{equation}
A subsolution $\iRc$ is said to be bounded if $\iRc\in\svbset^\bd$.
\end{definition}

It will be showed that the existence of subsolutions ``almost'' implies the solvability of problems~\ref{prob-a} and \ref{prob-b} (see Theorem~\ref{th3-01} below).
On the other hand, it is usually much easier to verify whether \eqref{2-01} has a subsolution.
These could help us to derive some explicit solvability conditions for SREs.
We remark that such a notion can be regarded as a stochastic counterpart of LMI proposed by Rami et al.~\cite{rami2001solvability} in their study of deterministic Riccati equations.

\section{Results}

The main results stated as the following two theorems are the basis of our further discussion.

\begin{theorem}\label{th3-00}
SLQ~$(\Cw, \Sw ,\Tw)$ is well-posed if SRE~$(\Cw, \Sw ,\Tw)$ has a bounded subsolution.
\end{theorem}

\begin{theorem}\label{th3-01}
Assume that there is a constant $\eps>0$ such that SRE~$(\Cw - \eps I_k, \Sw ,\Tw)$ has a bounded subsolution $\iRc$. 
Then

\emph{(i)} there exists a unique process $\Rc = \Rc(\cdot)$ in the following set 
\[
\svbset^{\bd}_{\ge\iRc} := \{K\in\svbset^\bd:K(t)\ge\iRc(t)\text{~almost surely~}\forall\,t\in[0,T]\}
\]
solving SRE~$(\Cw, \Sw ,\Tw)$;

\emph{(ii)} SLQ~$(\Cw, \Sw ,\Tw)$ is solvable; 
the value function $V(\xi) = \xi^\tr \Rc(0) \xi$,
and the unique optimal control
$u^*(t) = \Ctr(\Rc(t), \df{\Rc}(t)) x(t)$.
\end{theorem}

Theorem~\ref{th3-00} will be proved in Subsection~4.2.
The proof of Theorem~\ref{th3-01}, deferred to Subsection~4.3,
is based on an idea of Tang~\cite{tang2003general}, i.e., to represent a solution of SRE via the solution of a forward-backward stochastic differential equation (FBSDE).
The latter is usually called the generalized Hamiltonian system with respect to the associated control problem.

\begin{remark}\label{rmk.th}\rm
By Definition~\ref{def2-02} we can see that,
if $\Cw_1 \ge \Cw_2$, then each subsolution to SRE~$(\Cw_2, \Sw, \Tw)$
is also a subsolution to SRE~$(\Cw_1, \Sw, \Tw)$. 
Therefore, the assumption of Theorem~\ref{th3-01} can be stated equivalently as follows:
\emph{$\Cw \gg \hat{\Cw}$ and SRE~$(\hat{\Cw} , \Sw, \Tw)$ has a bounded subsolution.}
\end{remark}

\begin{remark}\rm
In Theorem~\ref{th3-01} the existence and uniqueness result for the SRE is restricted within a subset of $\svbset$ (namely $\svbset^{\bd}_{\ge\iRc}$),
which is natural and sensible from the point of view of optimal control.
Nevertheless, it is not clear so far whether the SRE admits a solution outside the set $\svbset^{\bd}_{\ge\iRc}$.
\end{remark}

The result in Theorem~\ref{th3-01} is, of course, not optimal.
A more satisfactory assertion might be ``an SRE is solvable if and only if it has a subsolution''; unfortunately, this is not true, even in the deterministic case.
Let us consider the following example.

\begin{example}\rm
Consider the following ODE over the time interval $[0,2]$:
\begin{equation*}
\dot{P} = \frac{P^2}{(1-t)^2\ch_{[0,1)}(t) + \ch_{[1,2]}(t)},
\quad P(2) = 1.
\end{equation*}
Clearly, $F= 0$ is a subsolution to this Riccati equation. However, it is easily verified that it has no continuous solution.
\end{example}

Nevertheless, this assertion would be true under some additional condition.
For instance, when the coefficients are all deterministic, the equation~\eqref{1-01a} subject to the stronger constraint
\begin{equation}\label{1-01c}
\Ch(\Rc) = \Cw + \isum D^\tr_i \Rc D_i \gg 0
\end{equation}
is solvable if and only if it has a subsolution satisfying \eqref{1-01c}; we thus conjecture that this may also be available for the stochastic case, but have not found any proof at the moment.

Next we derive from Theorem~\ref{th3-01} some explicit sufficient conditions that ensure the existence of solutions to SREs.
A basic idea is to consider the subsolutions with certain particular forms.
An interesting question is how ``negative'' the datum $\Cw$ could be to maintain the solvability of \eqref{2-01} when $\Sw$ and $\Tw$ are given.
Let us make a first attempt to this question.
In the following two results, we provide two ``robust'' criteria of the ``admissible'' $\Cw$.

In what follows, we denote 
\begin{equation}\label{2-10}
\minev(M) = \text{the minimal eigenvalue of a symmetric matrix }M.
\end{equation}
Note that, for a matrix-valued stochastic process $A$, $\minev(A)$ is a scalar stochastic process.

\begin{proposition}\label{prp3-01}
Let $\isum D_i^\tr D_i \gg 0$, and $\zeta:\citv{0,T} \to [0,1)$ be predictable.
Assume that the square-integrable predictable processes $(\vf,\psi_1,\dots,\psi_d)$ with $\vf>0$ satisfy the following BSDE:
\begin{equation}\label{2-07}
\md \vf = - \big[\minev(\Upsilon(\vf,\psi,\zeta)) \vf + \minev (\Sw)\big]\vd t
+ \psi\vd w_t, \quad \vf(T) = \minev(\Tw),
\end{equation}
where
\begin{gather*}
\Upsilon(\vf,\psi,\zeta) := A^\tr + A + \isum C_i^\tr C_i 
+ \isum \frac{\psi_i}{\vf} (C_i^\tr + C_i)\\
- \frac{1}{1-\zeta}
\Bigl(B + \isum C^\tr_i D_i + \isum \frac{\psi_i}{\vf}D_i\Bigr)
\Bigl(\isum D^\tr_i D_i\Bigr)^{\!\!-1}\!
\Bigl(B + \isum C^\tr_i D_i + \isum \frac{\psi_i}{\vf}D_i\Bigr)^{\!\!\tr}\!.
\end{gather*}
Then, SRE $(\Cw, \Sw ,\Tw)$ admits a bounded solution provided $\Cw \gg - \zeta\vf \isum D^\tr_i D_i$.
\end{proposition}

\begin{proof}
According to Remark~\ref{rmk.th}, 
it is sufficient to show that $F := \vf I_n$ is a subsolution to SRE $(\hat{\Cw}, \Sw ,\Tw)$ where $\hat{\Cw} := - \zeta\vf \isum D^\tr_i D_i$.  
Using the notation introduced in Definition~\ref{def2-03}, we have
\begin{align*}
\df{F}_i = \psi_i I_n,\quad  
\dr{F} = - \minev(\Upsilon(\vf,\psi,\zeta)) \vf I_n - \minev(\Sw) I_n.
\end{align*}
Then the expression $\func(F, \df{F})$ (recall \eqref{2-03}) associated to SRE~$(\hat{\Cw}, \Sw ,\Tw)$ reads
\begin{align*}
& \vf \Big( A^\tr + A + \isum C_i^\tr C_i \Big)
+ \isum \psi_i (C_i^\tr + C_i)
- \Big(\vf B + \vf \isum C^\tr_i D_i + \isum {\psi_i} D_i\Big) \\
& \times \Big[(-\zeta + 1)\vf \isum D^\tr_i D_i\Big]^{-1}
\Big(\vf B + \vf \isum C^\tr_i D_i + \isum {\psi_i} D_i\Big)^\tr + Q\\
& = \vf \Upsilon(\vf,\psi,\zeta) + Q.
\end{align*}
Keeping \eqref{2-10} in mind, since $\vf > 0$, we have
\begin{gather*}
\vf \Upsilon(\vf,\psi,\zeta) + Q
\ge \vf \minev(\Upsilon(\vf,\psi,\zeta)) I_n + \minev(Q) I_n
= - \dr{F}.
\end{gather*}
This along with the fact that 
$F(T) = \vf(T) I_n = \minev(\Tw) \le \Tw$ yields that $F$ is a subsolution to SRE~$(\hat{\Cw}, \Sw ,\Tw)$.
The proof is complete.
\end{proof}

Equation~\eqref{2-07} is actually a one-dimensional quadratic BSDE,
of which the existence of the solution was proved by Kobylanski \cite{kobylanski2000backward}.
Nevertheless, due to its high nonlinearity, \eqref{2-07} is often difficult to solve explicitly.
Therefore, we formulate a simplified version.
First of all, let us introduce another notation: for a matrix-valued random variable $M$, define
\[
\la_{\#}(M) = \essinf\nolimits_{\om\in\PS} \minev(M(\om)).
\]
Note that when $A$ is a process, $\la_{\#}(A)$ is a deterministic function of time variable.

\begin{theorem}\label{th3-02}
Let $\isum D_i^\tr D_i \gg 0$, and $\al:(0,T]\to[0,1)$.
Let $\vf>0$ satisfy the following ODE:
\begin{equation}\label{2-08}
\dot{\vf} + \la_{\#}(\varUpsilon(\al))\vf + \la_{\#}(\Sw) = 0,
\quad \vf(T) = \la_{\#}(\Tw),
\end{equation}
where
\begin{align*}
\varUpsilon(\al) ~:=~ & A^\tr + A + \isum C_i^\tr C_i - \frac{1}{1-\al}
\Big(B + \isum C^\tr_i D_i\Big)\\
&\cdot \Big(\isum D^\tr_i D_i\Big)^{-1}
\Big(B + \isum C^\tr_i D_i \Big)^\tr.
\end{align*}
Then, SRE $(\Cw, \Sw ,\Tw)$ admits a bounded solution provided $\Cw \gg - \al\vf \isum D^\tr_i D_i$.
\end{theorem}
\begin{proof}
The proof is analogous to that of Proposition~\ref{prp3-01} 
--- to show that $F := \vf I_n$ is a subsolution to SRE $(- \al\vf \isum D^\tr_i D_i, \Sw ,\Tw)$.
This is even simpler here as $\df{F} = 0$ in this case, so we omit the detail.
\end{proof}

An implicit condition of the above two results is that $\Tw\gg 0$.
Although the second criterion is rougher than the previous one,
it is significantly more feasible since \eqref{2-08} is a linear ODE that can be resolved explicitly as follows:
\begin{gather*}
\vf(t) = \varPhi(t,1)\la_{\#}(\Tw)
+ \int_t^1 \varPhi(t,s) \la_{\#}(\Sw(s)) \vd s
\quad\text{with}~~
\varPhi(t,s) = \me^{\int_t^s[\la_{\#}(\varUpsilon(\alpha))](r)\vd r}.
\end{gather*}
Since $\varPhi\gg 0$, an appropriate choice of $\la_{\#}(\Sw)$, even being negative, can also ensure $\vf > 0$.
Therefore, from the above result, one can easily construct various examples of solvable indefinite SREs, including those in which not only $R$ but also $Q$ is indefinite.
As far as we know, such a kind of solvability conditions seemed also new for the deterministic case.

Since only one-dimensional condition is concerned, this criterion is still rough, especially for multidimensional equations. 
Likely some refinement of the analysis will yield a more precise solvability condition,
for instance, $\al(\cdot)$ can be matrix-valued;
this is planned as future work. 
Nevertheless, the above result would be sharp in some one-dimensional cases.

\begin{example}\label{ex5-04}\rm
Consider the following equation
\begin{equation}\label{3-01}
\md P(t) = \frac{(P(t) + \Df(t))^2}{r(t)+P(t)}\vd t + \Df(t)\vd \BM_t,
\quad
r(t)+P(t) > 0,
\quad P(1) = 1.
\end{equation}
Take a function $\al:(0,1]\to[0,1)$. 
By Theorem~\ref{th3-02}, if
\begin{equation}\label{3-02}
r(t) > r_0(t) := - \al(t)\vf(t) = - \al(t) \exp\bigg(-\int_t^1\frac{1}{1-\al(s)}\vd s\bigg),
\end{equation}
then \eqref{3-01} admits a solution.
Indeed, how to choose the function $\al$ for different $r$ is tricky business.
Herein we consider, as an example, a special case that the threshold $r_0(t) = r_0$ is a constant, i.e.,
\[
\frac{\md r_0}{\md t}(t) = 0 ~\Longrightarrow~ \frac{\md \al}{\md t}(t) = - \frac{\al(t)}{1-\al(t)}.
\]
Thus, the inverse function of $\al$ is $t(\al) = \al - \ln \al + \con$ where $\con$ is a constant.
Evidently, $t(\cdot)$ is decreasing on $(0,1)$, and increasing on $(1,\infty)$, and $t(1) = 1+\con$.
To make $r_0 = - \al(1)$ as low as possible,
we choose $t(1) = 1+\con = 0$, i.e., $\con = -1$, then $\al(1)$ is a solution of the equation $1 = \al - \ln \al -1$; approximately, $\al(1) \approx 0.15859$.
Hence, \eqref{3-01} is solvable as long as $r(t) > - 0.15859$.
This value coincides with that given in~\cite[Example~3.2]{chen1998stochastic} where they considered deterministic equations.
\end{example}

In particular, Theorems~\ref{th3-01} and \ref{th3-02} yield directly the following known results.

\begin{corollary}\label{cor3-01}
Let $\Cw,\Sw,\Tw\ge 0$.
SRE $(\Cw, \Sw ,\Tw)$ admits a solution if, {\rm i)} $\Cw \gg 0$, or {\rm ii)} $\Tw \gg 0$ and $\isum D_i^\tr D_i \gg 0$.
\end{corollary}

The first case was an open problem proposed by Bismut~\cite{bismut1977controle} and Peng~\cite{peng1999open}, respectively, and resolved by Tang~\cite{tang2003general}.
The other was indicated by Kohlmann--Tang~\cite{kohlmann2003minimization,kohlmann2003multidimensional}.

Finally we extend a recent result of Qian--Zhou~\cite{qian2013existence},
where certain data of the SRE are not necessarily bounded.

\begin{proposition}\label{prp3-02}
Let $
(\sCw,\sSw) \in \ipr^2(\Sym^k\times\Sym^n)$ and 
$\sTw \in L^2(\PS,\Filt_T,\Sym^n)$.
Take $K\in \svbset$ such that 
\begin{align}\label{2-06}
K B + \isum (C_{i}^\tr K  D_{i} + \df{K}_{i} D_{i}) = 0.
\end{align}
Assume that $\hat{\Cw}$, $\hat{\Sw}$ and $\hat{\Tw}$ defined as below are all bounded and positive semi-definite:
\begin{align*}
& \hat{\Sw} = \sSw + (\dr{K}  + A^\tr K  + K  A) + \isum (C_{i}^\tr K  C_{i} + C_{i}^\tr \df{K}_{i} + \df{K}_{i} C_{i}),\\
& \hat{\Cw} = \sCw + \isum D^\tr_i K D_i,\\
& \hat{\Tw} = \sTw - K(T).
\end{align*}
Then, SRE~$(\sCw, \sSw ,\sTw)$ admits a solution $\sRc\in \svbset$
if either of the following cases occurs:
{\rm i)} $\hat{\Cw} \gg 0$;
{\rm ii)} $\hat{\Tw} \gg 0$ and $\isum D_i^\tr D_i \gg 0$.
Moreover, $\sRc - K$ is positive definite and bounded.
\end{proposition}
\begin{proof}
According to the assumptions, SRE~$(\hat{\Cw}, \hat{\Sw}, \hat{\Tw})$ admits a solution, denoted by $\hat{\Rc}$.
With \eqref{2-06} in mind, it is easily verified that $\sRc = \hat{\Rc} + K$ is a solution of SRE~$(\sCw, \sSw ,\sTw)$.
Moreover, $\sRc - K = \hat{\Rc}$ is positive definite and bounded.
The proof is complete.
\end{proof}

\begin{remark}\rm
Qian--Zhou~\cite{qian2013existence} introduced a direct approach to deal with a special case that $\hat{\Cw} = 0$ and $d=1$ (that means the Wiener process is one-dimensional).
Proposition~\ref{prp3-02} extends their results into great generality, and thus also recovers those obtained in \cite{hu2003indefinite} (see the comments in \cite[Section~5]{qian2013existence}).
We also note that,
the assumption that $\hat{\Sw}$ and $\hat{\Tw}$ are bounded did not appear in the statement of their main result, i.e., \cite[Theorem~2.2]{qian2013existence}, but was involved actually in their proofs, see Lemmas~3.1, 4.2 and 4.4 there;
in addition, they assumed the boundedness of $\sSw$, $\dr{K}$ and $\df{K}$.
\end{remark}

\section{Proofs}

\subsection{Auxiliary lemmas}

Let us first derive a basic a priori estimate for bounded solutions to SREs.
An analogous result has been obtained by Tang~\cite[Theorem~5.1]{tang2003general} for the definite case.

\begin{lemma}\label{lem4-06}
Let $\Rc\in\svbset^\bd$ be a solution to \eqref{2-01}.
Then, there is a generic constant $\con > 0$, depending only on $T$ and the bounds of $\Rc, A, C$ and $\Sw$, such that
\begin{equation}
\E\int_0^T\big(|\dr{\Rc}(t)| +  |\df{\Rc}(t)|^2\big)\vd t  \le \con.
\end{equation}
\end{lemma}
\begin{proof}
Recall \eqref{2-01} that
\begin{align*}
- \dr{\Rc} = \func(\Rc,\df{\Rc})
= \, &  A^\tr \Rc  + \Rc  A + \tsum_i (C_{i}^\tr \Rc  C_{i} + C_{i}^\tr \df{\Rc}_{i} + \df{\Rc}_{i} C_{i}) + \Sw\\
& - \Ctr(\Rc,\df{\Rc})^\tr \Ch(\Rc)\Ctr(\Rc,\df{\Rc}).
\end{align*}
Denote $P_\infty = \|\Rc\|_{L^\infty(\PS\times[0,T])} I_n$.
Applying It\^o's formula to $|\Rc + \Rc_\infty|^2$, we have
\begin{equation*}
\md |\Rc + \Rc_\infty|^2 = |\df{\Rc}|^2\vd t
+ 2\,\Tr[(\Rc + \Rc_\infty)\dr{\Rc}]\vd t
+ 2\isum \Tr[(\Rc + \Rc_\infty)\df{\Rc}_i]\vd \BM^i_t.
\end{equation*}
Taking expectation and by some standard arguments, we gain
\begin{align*}
\E\int_0^T |\df{\Rc}(t)|^2\vd t
\le \kappa 
+ \con\,\E \int_0^T \big\{ |\df{\Rc}|
- \Tr [(\Rc + \Rc_\infty)\Ctr(\Rc,\df{\Rc})^\tr \Ch(\Rc)\Ctr(\Rc,\df{\Rc})]\big\}(t)\vd t.
\end{align*}
Since $\Rc + \Rc_\infty \ge 0$ and $\Ch(\Rc) > 0$, 
\begin{equation*}
\begin{aligned}
& \Tr [(\Rc + \Rc_\infty)\Ctr(\Rc,\df{\Rc})^\tr \Ch(\Rc)\Ctr(\Rc,\df{\Rc})]\\
& = \Tr [(\Rc + \Rc_\infty)^{1/2}\Ctr(\Rc,\df{\Rc})^\tr \Ch(\Rc)\Ctr(\Rc,\df{\Rc})(\Rc + \Rc_\infty)^{1/2}]
\ge 0.
\end{aligned}
\end{equation*}
Thus, we get
\begin{align*}
\E\int_0^T |\df{\Rc}(t)|^2\vd t
\le \con
+ \con\,\E \int_0^T |\df{\Rc}(t)|\vd t
\le  \frac{1}{2}\,\E\int_0^T |\df{\Rc}(t)|^2\vd t +  \con,
\end{align*}
which yields the estimate for $\df{\Rc}$.
Finally, note that
\begin{equation*}
\begin{aligned}
0\le\;&\E\int_0^T \Ctr(\Rc,\df{\Rc})^\tr \Ch(\Rc)\Ctr(\Rc,\df{\Rc})(t)\vd t\\
\le\;& \E \Rc(T) - \Rc(0)
+ \E\int_0^T \!\!\Big[A^\tr \Rc  + \Rc  A + \tsum_i (C_{i}^\tr \Rc  C_{i} + C_{i}^\tr \df{\Rc}_{i} + \df{\Rc}_{i} C_{i}) + \Sw\Big]\!(t)\vd t\\
\le\;& \con + \con\, \E\int_0^T |\df{\Rc}(t)|^2\vd t \le \con.
\end{aligned}
\end{equation*}
This yields the estimate for $\dr{\Rc}$.
The proof is complete.
\end{proof}

The following result is a key step toward Theorem~\ref{th3-01}, which indicates that a solution of the SRE can be constructed from the solution of a forward-backward stochastic differential equation (FBSDE) system, provided that the latter exists and satisfies some appropriate conditions.

\begin{lemma}\label{lem4-01}
Assume that the following FBSDE system
\begin{equation}\label{4-01}
\left\{\begin{aligned}
& \md X = (AX + BU)\vd t + \isum(C_iX + D_iU)\vd \BM^i_t,
\\
& \md Y = -(A^\tr Y + \isum C_i^\tr Z_i + Q X)\vd t + \isum Z_i\vd \BM^i_t,
\\
& 0 = \Cw U + B^\tr Y + \isum D_i^\tr Z_i,\\
& X(0) = I_n,\quad Y(T) = \Tw X(T)
\end{aligned}\right.
\end{equation}
has a solution
\[
(X,U,Y,Z)\in \cpr^2(\R^{n\times n})\times\ipr^2(\R^{k\times n})\times\cpr^2(\R^{n\times n})\times(\ipr^2(\R^{n\times n}))^d,
\]
moreover, there are a process $K\in\svbset^\bd$ and a constant $\con\in\R_+$ such that 
\begin{equation}\label{4-21}
\begin{gathered}
X^\tr K X \le X^\tr Y \le \con X^\tr X,\\
\Cw + \isum D^\tr_i K D_i \gg 0.
\end{gathered}
\end{equation}
Then, $X^{-1} = \{X(t)^{-1};\,t\in[0,T]\}$ has a continuous version, and
\begin{equation}\label{4-04}
\begin{aligned}
\Rc = YX^{-1}\in\svbset^{\bd}
\end{aligned}
\end{equation}
is a solution of SRE~$(\Cw,\Sw,\Tw)$ with 
\begin{equation}\label{4-04a}
\begin{aligned}
\df{\Rc}_i= Z_i X^{-1} - YX^{-1}(C_i + D_i U X^{-1}),
\end{aligned}
\end{equation}
and $K(t)\le \Rc(t) \le \con I_n$ a.s. for each $t\in [0,T]$.
\end{lemma}

\begin{remark}\rm
The process $X^\tr Y$ takes values in $\Sym^n$.
Indeed, using It\^o's formula to $X^{\tr}Y$, we have
\begin{equation*}
X(t)^{\tr}Y(t) = \E^{\Filt_t} 
\biggl[
\int_t^T \big( U^\tr \Cw U + X^\tr \Sw X \big)(r)\vd r
+ X(T)^{\tr} \Tw X(T)
\biggr],\quad\forall\,t\in[0,T].
\end{equation*}
Clearly, the right-hand side is an $\Sym^n$-valued random variable.
\end{remark}

Before the rigorous proof, let us do some heuristic computations.
Suppose $X^{-1}$ exists.
Set
\begin{equation}\label{4-03}
\begin{aligned}
\Rc = YX^{-1},\quad
\Df_i = Z_i X^{-1} - YX^{-1}(C_i + D_i U X^{-1}).
\end{aligned}
\end{equation}
By use of the fact that $\md(XX^{-1}) = 0$, we derive the equation of $X^{-1}$ as 
\begin{align*}
\md (X^{-1}) = &
- X^{-1} \Big[ A + B U X^{-1} - \isum (C_i + D_i U X^{-1})^2 \Big] \vd t \\
& - X^{-1}\isum (C_i + D_i U X^{-1})\vd \BM^i_t.
\end{align*}
Thus, with $Z_i X^{-1} = \Df_i + \Rc(C_i + D_i U X^{-1})$ in mind,
we gain that
\begin{equation}\label{4-05}
\begin{aligned}
\md \Rc = \; & \md (YX^{-1}) = - \isum Z_i X^{-1}(C_i + D_i U X^{-1})\vd t
+ Y\vd (X^{-1}) + (\md Y)X^{-1}\\
= \; & - \isum \big[ Z_i X^{-1} - YX^{-1}(C_i + D_i U X^{-1})\big](C_i + D_i U X^{-1})\vd t\\
& - \Big[ A^\tr YX^{-1} + \Sw + \isum C_i^\tr Z_i X^{-1} \Big]\vd t
- YX^{-1} (A + B U X^{-1})\vd t
\\
& + \isum \big[ Z_i X^{-1} - YX^{-1}(C_i + D_i U X^{-1})\big]\vd \BM^i_t\\
=\; & - \Big[ A^\tr P + P A + \isum (C_i^\tr \Rc C_i + C_i^\tr \Df_i + \Df_i C_i) + \Sw
\Big]\vd t\\
& - \Big[\Rc B + \isum (C_i^\tr \Rc D_i + \Df_i D_i)\Big] U X^{-1}\vd t
+ \isum \Df_i \vd \BM^i_t.
\end{aligned}
\end{equation}
On the other hand, it follows from \eqref{4-01} that
\begin{equation*}
\begin{aligned}
0 = \;& \Cw U X^{-1} + B^\tr Y X^{-1} + \isum D_i^\tr Z_i X^{-1}\\
=\;& \Cw U X^{-1} + B^\tr \Rc + \isum D_i^\tr[\Df_i + \Rc(C_i + D_i U X^{-1})]\\
=\;& \Big(\Cw + \isum D_i^\tr \Rc D_i\Big)UX^{-1}
+ B^\tr \Rc + \isum (D_i^\tr\Rc C_i + D_i^\tr \Df_i);
\end{aligned}
\end{equation*}
if $\Cw + \isum D_i^\tr \Rc D_i > 0$, then
\begin{equation}\label{4-06}
UX^{-1} = - \Big(\Cw + \isum D_i^\tr \Rc D_i\Big)^{-1}
\Big[B^\tr \Rc + \isum (D_i^\tr\Rc C_i + D_i^\tr \Df_i)\Big].
\end{equation}
Substituting \eqref{4-06} into \eqref{4-05}, one can find that $(\Rc,\Df)$ defined in \eqref{4-03} satisfies \eqref{1-01a} formally.

\begin{remark}\label{rem4-01}\rm
From \eqref{4-06}, the equation of $X$ can be rewritten as
\begin{equation}\label{4-08}
\md X = [ A + B\varGamma(\Rc,\Df)]X\vd t
+ \isum [ C_i + D_i\varGamma(\Rc,\Df)]X\vd \BM^i_t,\quad
X(0) = I_n,
\end{equation}
as long as $(\Rc,\Df)$ is well-defined, where $\varGamma(\Rc,\Df)$ is defined in \eqref{2-03}.
\end{remark}

We are now in a position to prove Lemma~\ref{lem4-01}. 
The key point, suggested by the above heuristic analysis, is to show the existence and continuity of the reciprocal process of $X$.

\emph{Proof of Lemma~\ref{lem4-01}}.
First of all, $\Rc(0) = Y(0)$ is well-defined.
Recall \eqref{2-10} the definition of $\minev(\cdot)$, and introduce the stopping times:
\begin{align*}
\st_m & = \inf\Big\{t: \minev\big(X^\tr(t)X(t)\big) \le \frac{1}{m}\Big\} \wedge T,
\quad m\in\mathbb{N}^*,\\
\st & = \st_{\infty} = \inf\big\{t: \minev\big(X^\tr(t)X(t)\big) \le 0\big\} \wedge T.
\end{align*}
Clearly, $\st_m \uparrow \st$,
and $X^{-1}$ exists on the set $\itv{0,\st}$, and is bounded on $\itv{0,\st_m}$; $(\Rc,\Df)$ is thus well-defined on $\itv{0,\st}$.
Keeping in mind \eqref{4-21}, we have
\begin{equation}
[X^{\tr}KX](t,\om) 
\le [X^{\tr}\Rc X] (t,\om)
\le \con [X^{\tr}X](t,\om) 
\quad\forall\, (t,\om) \in \itv{0,\st},
\end{equation}
thus,
\begin{equation}\label{4-23}
K(t,\om) 
\le \Rc(t,\om)
\le \con I_n
\quad\forall\, (t,\om) \in \itv{0,\st},
\end{equation}
and furthermore, $\Cw + \isum D_i^\tr \Rc D_i \gg 0$ on $\itv{0,\st}$.
Define
\begin{equation*}
\Rc^{(m)}:= \Rc\ch_{\citv{0,\st_m}} + \Rc(\st_m)\ch_{\oitv{\st_m,T}},
\quad
\Df^{(m)} := \Df\ch_{\citv{0,\st_m}}. 
\end{equation*}
Clearly, $\breve{\Rc}^{(m)} = \Df^{(m)}$.
According to our heuristic computations, $\Rc^{(m)}$ is a bounded solution to 
\[
\text{SRE}~(A^{(m)},B^{(m)},C^{(m)},D^{(m)};\,
\Cw^{(m)}, \Sw^{(m)}, \Tw^{(m)})
\]
with
\begin{equation*}
\begin{array}{llll}
A^{(m)} := A\ch_{\citv{0,\st_m}},
& B^{(m)} := B\ch_{\citv{0,\st_m}},
& C^{(m)} := C\ch_{\citv{0,\st_m}},
& D^{(m)} := D\ch_{\citv{0,\st_m}},\\
\Cw^{(m)} := \Cw,
& \Sw^{(m)} := \Sw\ch_{\citv{0,\st_m}},
& \Tw^{(m)} := \Rc(\st_m).&
\end{array}
\end{equation*}
Define, as in \eqref{2-03}, the corresponding
\begin{equation*}
\Ch^{(m)}(\Rc),\quad 
\Ctr^{(m)}(\Rc,\Df),\quad
\func^{(m)}(\Rc,\Df).
\end{equation*}
By means of Lemma~\ref{lem4-06}, there is a constant $\con_1$ independent of $m$ such that
\[
\E \int_0^T \big(|\Df^{(m)}(t)|^2 + |\func^{(m)}(\Rc^{(m)}(t),\Df^{(m)}(t))| \big) \vd t \le \con_1 <\infty.
\]
Since, as $m\to\infty$,
\begin{equation*}
\Df^{(m)} \to \Df \ch_{\itv{0,\st}},\quad
\func^{(m)}(\Rc^{(m)},\Df^{(m)}) \to
\func(\Rc,\Df)\ch_{\itv{0,\st}} 
\quad\text{a.e. on }
\citv{0,T},
\end{equation*}
by Fatou's lemma we have
\[
\E \int_0^T \ch_{\itv{0,\st}}(t) \big(|\Df(t)|^2 + |\func(\Rc(t),\Df(t))| \big) \vd t\le \con_1 <\infty.
\]
Since
\begin{equation*}
|\Df\ch_{\itv{\st_m,\st}}|\le|\Df\ch_{\itv{0,\st}}|,\quad
|\func(\Rc(t),\Df(t))\ch_{\itv{\st_m,\st}}|\le|\func(\Rc(t),\Df(t))\ch_{\itv{0,\st}}|,
\end{equation*}
it follows from Lebesgue's dominated convergence theorem that
\begin{align*}
& \lim_{m\to \infty}\E \int_0^T \ch_{\itv{0,\st}}(t)|\Df^{(m)}(t) - \Df(t)|^2 \vd t
= 0,\\
& \lim_{m\to \infty}\E \int_0^T \ch_{\itv{0,\st}}(t) |\func^{(m)}(\Rc^{(m)}(t),\Df^{(m)}(t))-\func(\Rc(t),\Df(t))| \vd t
= 0.
\end{align*}
Therefore, we obtain
\begin{gather*}
\int_0^T \isum \Df^{(m)}_i(t)\vd \BM^i_t \to
\int_0^T \isum \ch_{\itv{0,\st}} \Df_i(t)\vd \BM^i_t\quad\text{a.s.},\\
\int_0^T \func^{(m)}(\Rc^{(m)}(t),\Df^{(m)}(t))\vd t \to
\int_0^T \ch_{\itv{0,\st}} \func(\Rc(t),\Df(t))\vd t\quad\text{a.s.}
\end{gather*}
Define
\begin{equation*}
\Tw^{(\infty)} := \Rc(0) - \int_0^T \ch_{\itv{0,\st}} \func(\Rc(t),\Df(t))\vd t
+ \int_0^T \isum \ch_{\itv{0,\st}} \Df_i(t)\vd \BM^i_t,
\end{equation*}
and
\begin{equation*}
\Rc^{(\infty)} := \Rc\ch_{\itv{0,\st}} + \Tw^{(\infty)}\ch_{\citv{\st,T}},
\quad
\Df^{(\infty)} :=\Df\ch_{\itv{0,\st}},
\end{equation*}
and
\begin{equation*}
\begin{array}{llll}
A^{(\infty)} := A\ch_{\citv{0,\st}},
& B^{(\infty)} := B\ch_{\citv{0,\st}},
& C^{(\infty)} := C\ch_{\citv{0,\st}},\\
D^{(\infty)} := D\ch_{\citv{0,\st}},
& \Cw^{(\infty)} := \Cw,
& \Sw^{(\infty)} := \Sw\ch_{\citv{0,\st}},&
\end{array}
\end{equation*}
and the corresponding
\begin{equation*}
\Ch^{(\infty)}(\Rc),\quad 
\Ctr^{(\infty)}(\Rc,\Df),\quad
\func^{(\infty)}(\Rc,\Df).
\end{equation*}
Then $\Rc^{(\infty)}$ solves SRE $(A^{(\infty)},B^{(\infty)},C^{(\infty)},D^{(\infty)}; \Cw^{(\infty)},\Sw^{(\infty)},\Tw^{(\infty)})$, with $\df{\Rc}^{(\infty)} = \Df^{(\infty)}$;
moreover, by the trajectory-continuity of $\Rc^{(\infty)}$ we know $\Tw^{(\infty)}$ is bounded, that means $\Rc^{(\infty)} \in \svbset^\bd$.
Also, remember that 
\[\Ch^{(\infty)}(\Rc^{(\infty)}) = \Cw^{(\infty)} + \isum (D^{(\infty)}_i)^\tr \Rc^{(\infty)} D^{(\infty)}_i \gg 0,\]
thus $\varGamma^{(\infty)}(\Rc^{(\infty)},\Df^{(\infty)})\in \ipr^2(\R^{k\times n})$.

Let us consider the following SDE over the time horizon $[0,T]$:
\begin{equation}\label{4-07}
\begin{aligned}
\md X^{(\infty)} =\;& (A^{(\infty)} + B^{(\infty)}\varGamma^{(\infty)}(\Rc^{(\infty)},\Df^{(\infty)}))X^{(\infty)}\vd t\\
& + \isum (C_i^{(\infty)}
+ D_i^{(\infty)}\varGamma^{(\infty)}(\Rc^{(\infty)},\Df^{(\infty)}))X^{(\infty)}\vd \BM^i_t,\\
X^{(\infty)}(0) =\;& I_n.
\end{aligned}
\end{equation}
We need the following result whose proof will be given later.

\begin{lemma}\label{lem4-04}
Let $\vA,\vC_i$ $(i=1,\dots,d)$ be $\R^{n\times n}$-valued adapted processes such that 
$$\int_0^\infty \Big(|\vA(t)|+\isum|\vC_i(t)|^2\Big)\vd t < \infty\quad\text{a.s.}$$
Then, the following SDE
\begin{align}\label{4-25}
\md \vX = \vA \vX \vd t + \isum \vC_i \vX \vd \BM^i_t,\quad
\vX(0) = M\in \R^{n\times m}
\end{align}
has a unique strong solution. 
Moreover, when $m=n$ and $M=I_n$, $\vX^{-1} = \{\vX(t)^{-1};t\ge 0\}$ 
exists and is a continuous process.
\end{lemma}

By means of Lemma~\ref{lem4-04}, $(X^{(\infty)})^{-1}$ is a continuous process, thus $X^{(\infty)}(\tau)$ is invertible a.s.
On the other hand, in view of Remark~\ref{rem4-01}, \eqref{4-07} coincides with \eqref{4-08} on $\itv{0,\st}$.
Thus $X^{(\infty)} = X$ a.e. on $\itv{0,\st}$.
By the trajectory-continuity of solutions of SDEs,
\[
X(\st) = X^{(\infty)}(\st)\quad\text{a.s.},
\]
so $X(\st)$ is invertible a.s. 
Recalling the definition of $\st$, we gain that
\[
\Prob(\tau = T) = 1.
\]
Therefore, $X^{-1} = (X^{(\infty)})^{-1}$ is a continuous process; $(\Rc,\Df)$ given in \eqref{4-03} is then well-defined on $\citv{0,T}$, solving SRE \eqref{1-01}. Clearly, $\df{\Rc}= \Df$,
and from \eqref{4-23} we have $K\le \Rc \le \con I_n$.
The proof of Lemma~\ref{lem4-01} is complete.
\quad\ensuremath{\square}

\emph{Proof of Lemma~\ref{lem4-04}.}
The existence and uniqueness of the strong solution to \eqref{4-25} follows from a well-known result due to Gal'chuk~\cite[basic theorem on pp. 756--757]{gal1979existence} (see~\cite[Lemma~7.1]{tang2003general} for more related formulation). It remains to show the invertibility of $\vX$ when $m=n$ and $M=I_n$.
Note that the SDE
\begin{align*}
\md \vvX = - \vvX\Big(\vA - \isum \vC_i\vC_i\Big) \vd t - \vvX\isum \vC_i  \vd \BM^i_t,\quad
\vvX(0) = I_n
\end{align*}
also has a unique (continuous) strong solution.
Then, $V = \vX\vvX$ satisfies
\begin{align*}
\md V = \Big[ \vA V - V\vA + \isum (V \vC_i - \vC_i V)\vC_i \Big]\vd t + \isum (\vC_i V - V\vC_i) \vd \BM^i_t,\quad
V(0) = I_n.
\end{align*}
The uniqueness of the solution implies $V=I_n$,
that means $\vX^{-1} = \vvX$.\quad\ensuremath{\square}

The following lemma and its proof collect some computations that are useful in the proofs of our main results.

\begin{lemma}\label{lem4-10}
Let $\iRc \in \svbset^\bd$ be a subsolution to SRE~$(\Cw - \eps I_k, \Sw, \Tw)$ with $\eps \ge 0$, 
and $x=x(\cdot)$ be the solution to \eqref{1-004b} with $\xi\in\R^n$ and $u\in \ipr^2(\R^k)$.
Then
\begin{equation*}
\xi^\tr \iRc(0) \xi + \eps\,\E \int_0^{T} |u(t)|^2\vd t \le J(u;\xi),
\end{equation*}
where $J(u;\xi)$ is defined in \eqref{cost}.
\end{lemma}

\begin{proof}
It follows from It\^o's formula that
\begin{equation}\label{4-11}
\begin{aligned}
\md (x^\tr \iRc  x) =\;& 
x^\tr \Big[ \dr{\iRc } + A^\tr \iRc  + \iRc A + \isum (C^\tr_i \iRc  C_i + C^\tr_i \df{\iRc }_i + \df{\iRc }_i C_i) \Big] x \vd t\\
& + x^\tr \Big[ \iRc B + \isum (C^\tr_i \iRc + \df{\iRc }_i) D_i \Big] u \vd t\\
& + u^\tr \Big[ B^\tr \iRc  + \isum  D_i^\tr(\iRc C_i + \df{\iRc }_{i}) \Big] x
+ \isum u^\tr D^\tr_i \iRc D_i u \vd t\\
& + \isum \Big[ x^\tr (\df{\iRc }_{i} + C_i^\tr \iRc  + \iRc  C_i )x
+ x^\tr \iRc  D_i u + u^\tr D^\tr_i \iRc  x  \Big]\vd \BM^i_t.
\end{aligned}
\end{equation}
Since $\iRc\in\svbset^\bd$ is a subsolution to SRE~$(\Cw - \eps I_k, \Sw ,\Tw)$, i.e.,
\begin{equation}\label{4-24}
\begin{aligned}
- \dr{\iRc } \le \func_\eps :=\;& A^\tr \iRc  + \iRc A + \isum (C^\tr_i \iRc  C_i + C^\tr_i \df{\iRc }_i + \df{\iRc }_i C_i) + \Sw\\
& - \Big[ \iRc B + \isum (C^\tr_i \iRc + \df{\iRc }_i) D_i \Big]\Big( \Cw - \eps I_k + \isum D^\tr_i \iRc D_i\Big)^{-1} \\
& \times\Big[ B^\tr \iRc  + \isum  D_i^\tr(\iRc C_i + \df{\iRc }_i) \Big],
\end{aligned}
\end{equation}
by the method of completing the square, we can derive from \eqref{4-11} that
\begin{equation}\label{4-12}
\begin{aligned}
\md (x^\tr \iRc  x) =\;& (\dots)\vd \BM_t - (u^\tr \Cw u + x^\tr \Sw x \big)\vd t + \big[\eps |u|^2 + x^\tr ( \dr{\iRc } + \func_\eps ) x\big]\vd t \\
&  + (u - \varGamma_\eps x)^\tr \Ch_{\eps} (u - \varGamma_\eps x)\vd t,
\end{aligned}
\end{equation}
where
\begin{equation*}
\begin{aligned}
\Ch_\eps := \;& \Cw - \eps I_k + \isum D^\tr_i \iRc D_i > 0,\\
\varGamma_\eps := \;& \Ch_\eps^{-1} \Big[ B^\tr \iRc  + \isum  D_i^\tr(\iRc C_i + \df{\iRc }_i) \Big].
\end{aligned}
\end{equation*}
Because it is not clear whether the It\^o integral term is a martingale, we define the stopping times:
\[
\si_{\! m} = \inf\{t: |x(t)| \ge m\}\wedge T.
\]
Clearly, $\si_{\! m} \uparrow T$.
Note that $x(\cdot)$ is bounded on $\citv{0,\si_{\! m}}$.
Integrating \eqref{4-12} on $\citv{0,\si_{\! m}}$, we have
\begin{equation}\label{4-13}
\begin{aligned}
\E & \int_0^{\si_{\! m}}\!\! \big( u^\tr \Cw u + x^\tr \Sw x  \big)(t)\vd t
+ \E\big[ x(\si_{\! m})^\tr \iRc(\si_{\! m}) x(\si_{\! m}) \big]\\
= & \; \xi^\tr \iRc(0) \xi + \E \int_0^{\si_{\! m}}\!\! \big[\eps |u|^2 + x^\tr (\dr{\iRc } +\func_\eps) x + (u - \varGamma_\eps x)^\tr \Ch_{\eps} (u - \varGamma_\eps x) \big](t)\vd t.
\end{aligned}
\end{equation}
On one hand, by the trajectory-continuity of $x$ and $\iRc$, and Lebesgue's dominated convergence theorem, we get
\[
\lim_{m\to\infty} \E\big[x(\si_{\! m})^\tr \iRc(\si_{\! m}) x(\si_{\! m})\big] = \E\big[x(T)^\tr \iRc(T) x(T)\big]
\le\E\big[ x(T)^\tr \Tw x(T)\big].
\] 
On the other hand, we know
\[
x^\tr (\dr{\iRc } +\func_\eps) x + (u - \varGamma_\eps x)^\tr \Ch_{\eps} (u - \varGamma_\eps x)
\ge 0.
\]
Therefore, letting $m\to\infty$ in \eqref{4-13}, we have 
\begin{equation*}
\begin{aligned}
\xi^\tr \iRc(0) \xi + \eps\,\E \int_0^{T} |u(t)|^2\vd t
\le \E\bigg[ x(T)^\tr \Tw x(T) 
+ \int_0^{T}\!\! \big( u^\tr \Cw u + x^\tr \Sw x  \big)(t)\vd t\bigg].
\end{aligned}
\end{equation*}
The lemma is proved.
\end{proof}

\subsection{Proof of Theorem~\ref{th3-00}}

Let $\iRc\in\svbset^{\bd}$ be a subsolution of SRE~$(\Cw,\Sw,\Tw)$.
Then applying Lemma~\ref{lem4-10} with $\eps = 0$ and $\xi=0$,
we know that $J(u;0) \ge 0$ for any $u\in \ipr^2(\R^k)$, thus $V(0) \ge 0$, which concludes Theorem~\ref{th3-00}.

\subsection{Proof of Theorem~\ref{th3-01}}

The proof is divided into the following four steps.

\emph{Step 1.}
We shall prove that SLQ~$(\Cw,\Sw,\Tw)$ is solvable.

The argument is similar to \cite[the proof of Theorem~3.1]{bismut1976linear}.
Fix $\xi\in\R^n$.
Recalling \eqref{4-13}, its right-hand side is convex in $u$ and $x$, while $x$ is linear in $u$, thus
\[
\E \int_0^{\si_{\! m}} \big( u^\tr \Cw u + x^\tr \Sw x  \big)(t)\vd t
+ \E\big[ x(\si_{\! m})^\tr \iRc(\si_{\! m}) x(\si_{\! m}) \big]
\]
is convex in $u$.
Sending $m\to \infty$ implies that $J(u;\xi)$ is convex in $u$.
Clearly, $J(\,\cdot\,;\xi)$ is continuous on $\ipr^2(\R^k)$.
Moreover, when $\|u\|_{L^2}\to \infty$,
$J(u;\xi) \to +\infty$ by Lemma~\ref{lem4-10},
that implies, when $\al$ is large enough, $\{u: J(u)\le \al\}$ is convex and weakly compact.
Then from a well-known result (cf.~\cite[Proposition~2.1.2]{ekeland1999convex}), $J(\,\cdot\,;\xi)$ has an optimum.

\emph{Step 2.}
We shall prove the existence of the solution to SRE~$(\Cw,\Sw,\Tw)$.

To apply Lemma~\ref{lem4-01}, let us first prove the following result.

\begin{lemma}\label{lem4-05}
Under the assumption of Theorem~\ref{th3-01}, for any 
$s\in[0,T)$ and $\xi\in L^2(\PS,\Filt_s,\R^n)$,
the following FBSDE system
\begin{equation}\label{4-09}
\left\{\begin{aligned}
& \md x^{s,\xi}  = (Ax^{s,\xi} + Bu^{s,\xi})\vd t + \isum(C_ix^{s,\xi} + D_iu^{s,\xi})\vd \BM^i_t,\\
& \md y^{s,\xi}  = -(A^\tr y^{s,\xi} + \isum C_i^\tr z^{s,\xi}_i + Q x^{s,\xi})\vd t + \isum z^{s,\xi}_i\vd \BM^i_t,\\
& \Cw u^{s,\xi}  + B^\tr y^{s,\xi} + \isum D_i^\tr z^{s,\xi}_i = 0,\\
& x^{s,\xi}(s)  = \xi,\quad
y^{s,\xi}(T) = \Tw x^{s,\xi}(T)
\end{aligned}\right.
\end{equation}
admits a unique solution $(x^{s,\xi},u^{s,\xi},y^{s,\xi},z^{s,\xi})$ such that
\begin{equation}\label{4-20}
x^{s,\xi},y^{s,\xi}\in \cpr^2(s,T;\R^n),\quad
z^{s,\xi} \in \ipr^2(s,T;\R^{n\times d}),\quad
u^{s,\xi} \in \ipr^2(s,T;\R^k);
\end{equation}
moreover, 
\begin{equation}\label{4-26}
\xi^\tr y^{s,\xi}(s) \ge \xi^\tr \iRc(s)\xi\quad\text{a.s.},
\end{equation}
and 
there is a generic constant $\con_0 > 0$, depending only on $\eps,T, \iRc,A,B,C,D,\Cw,\Sw$ and $\Tw$, such that
\begin{equation}\label{4-10}
\begin{aligned}
\E\biggl[\sup_{t\in[s,T]} \!|x^{s,\xi}(t)|^2 
+ \sup_{t\in[s,T]} \!|y^{s,\xi}(t)|^2 + \int_s^T\!\!\! \big(|u^{s,\xi}(t)|^2 + |z^{s,\xi}(t)|^2\big) \md t\biggr]
\le \con_{0} \E |\xi|^2.
\end{aligned}
\end{equation}
\end{lemma}

\begin{proof}
For simplicity, we present the details only for the case that $s = 0$ and $\xi\in\R^n$;
the argument also works for the general case.
Write $(x,u,y,z) = (x^{s,\xi},u^{s,\xi},y^{s,\xi},z^{s,\xi})$ simply.

From the first step of the proof of Theorem~\ref{th3-01}, $J(\,\cdot\,;\xi)$ has an optimum, denoted by $u^*$.
Let $x^*$ be the solution of \eqref{1-004b} with respect to $u^*$.
Then by means of the stochastic maximum principle (cf. \cite[Section 3.1]{bismut1978introductory}), the optimal solution $(x^*, u^*)\in\cpr^2\times\ipr^2$ along with its adjoint processes satisfies a generalized Hamiltonian system that coincides with \eqref{4-01},
thus the existence is obtained.

Next we derive the estimate \eqref{4-10}.
Applying It\^o's formula to $x^\tr y$ and proceeding some standard arguments, we have ($\delta > 0$)
\begin{equation}\label{4-15}
\begin{aligned}
J :=\;& \E \int_0^T \!\!\big( u^\tr \Cw u + x^\tr \Sw x  \big)(t)\vd t
+ \E\big[ x(T)^\tr \Tw x(T) \big]\\
=\;& \xi^\tr  y(0) 
\le \delta \,|y(0)|^2
+ 4\delta^{-1} |\xi|^2.
\end{aligned}
\end{equation}
Lemma~\ref{lem4-10} yields that
\begin{equation}\label{4-50}
\xi^\tr \iRc(0) \xi + \eps\,\E \int_0^{T} |u(t)|^2\vd t \le J.
\end{equation}
Moreover,
it follows from classical estimates for SDEs and BSDEs (cf. \cite{el1997backward}) that
\begin{equation}\label{4-16}
\begin{gathered}
\E \sup_{t\in[0,T]} |x(t)|^2 \le \con_{0} \bigg[|\xi|^2 + \E\int_0^T |u(t)|^2\vd t\bigg],\\
\E \bigg[ \sup_{t\in[0,T]} |y(t)|^2 + \int_0^T  |z(t)|^2 \vd t\bigg]
\le \con_{0}\, \E \sup_{t\in[0,T]} |x(t)|^2.
\end{gathered}
\end{equation}
Combining \eqref{4-15}, \eqref{4-50} and \eqref{4-16}, and taking the positive number $\delta$ sufficiently small, we gain the estimate \eqref{4-10}.

The uniqueness of the solution follows from \eqref{4-10};
moreover, from \eqref{4-15} and \eqref{4-50}, we have $\xi^\tr y(0) \ge \xi^\tr \iRc(0) \xi$,
thus \eqref{4-26} is derived.
The proof of Lemma~\ref{lem4-05} is completed.
\end{proof}

Let us move on the proof of Theorem~\ref{th3-01}.
The existence of the solution of~\eqref{4-01} follows from Lemma~\ref{lem4-05}.
By the uniqueness, we know that
\begin{equation*}
Y(t)\eta = (y^{s,X_1(s)\eta}(t),\dots,y^{s,X_n(s)\eta}(t))
\quad
\text{a.s.}~~\forall\,t\in[s,T],\,\eta\in L^\infty(\PS,\Filt_s,\R^n),
\end{equation*}
where $X_i(s)$, $i=1,\dots,d$, is the $i$-th column vector of $X(s)$,
so \eqref{4-15} yields 
\[
\E[\eta^\tr X(s)^\tr Y(s)\eta] \le \con_0 \,\E[\eta^\tr X(s)^\tr X(s)\eta],
\]
that implies
\[
X(s)^{\tr}Y(s) \le \con_0\, X(s)^{\tr}X(s)
\quad
\text{a.s.}~~\forall\,s\in[0,T].
\]
Analogously, from \eqref{4-26} we can derive
\[
X(s)^{\tr}Y(s) \ge X(s)^{\tr}\iRc(s)X(s)
\quad\text{a.s.}~~\forall\,s\in[0,T].
\]
Since $\iRc$ is a bounded subsolution to SRE~$(\Cw-\eps I_k,\Sw,\Tw)$,
the condition~\eqref{4-21} is then satisfied by taking $K = \iRc$ and $\con = \con_0$.
Therefore, by means of Lemma~\ref{lem4-01}, SRE~$(\Cw,\Sw,\Tw)$ admits a solution in the set $\svbset^{\bd}_{\ge \iRc}$.

\emph{Step 3.}
We shall prove that SRE~$(\Cw,\Sw,\Tw)$ has at most one solution in the set $\svbset^{\bd}_{\ge\iRc}$.

Let $P\in\svbset^{\bd}_{\ge\iRc}$ be a solution to SRE~$(\Cw,\Sw,\Tw)$.
From Lemma~\ref{lem4-04}, the following SDE
\begin{equation}
\md x(t) = [A + B\Ctr(P,\df{P})]x(t) \vd t
+ \isum [C_i + D_i \Ctr(P,\df{P})]x(t) \vd \BM^i_t,
\quad
x(0)=\xi\in\R^n
\end{equation}
has a unique strong solution $x_{P} = x_P(\cdot)$.
Denote 
\[u_{P}(t) = \Ctr(P(t),\df{P}(t)) x_{P}(t),\]
and define the stopping times
$\sigma_m = \inf\{t:|x(t)|\ge m\}\wedge T$.
Then from~\eqref{4-13} (with $\eps = 0$ and $P$ instead of $\iRc$) we have
\begin{equation}\label{4-60}
\xi^\tr P(0) \xi =
\E \int_0^{\si_{\! m}}\!\! \big( u_{P}^\tr \Cw u_{P} + x_{P}^\tr \Sw x_{P}  \big)(t)\vd t
+ \E\big[ x_{P}(\si_{\! m})^\tr P(\si_{\! m}) x_{P}(\si_{\! m}) \big].
\end{equation}
On the other hand, since $\iRc$ is a bounded subsolution to SRE~$(\Cw - \eps I_k,\Sw,\Tw)$,
it follows from~\eqref{4-13} that
\begin{align*}
& \xi^\tr \iRc(0) \xi + \eps \E \int_0^{\si_{\! m}}\!\! |u_{P}(t)|^2 \vd t \\
& \quad \le
\E \int_0^{\si_{\! m}}\!\! \big( u_{P}^\tr \Cw u_{P} + x_{P}^\tr \Sw x_{P}  \big)(t)\vd t
+ \E\big[ x_{P}(\si_{\! m})^\tr \iRc(\si_{\! m}) x_{P}(\si_{\! m}) \big].
\end{align*}
Comparing the last two formulae, and keeping in mind $P(\si_{\! m})\ge\iRc(\si_{\! m})$, we have
\[
\eps\, \E \int_0^{\si_{\! m}}\!\! |u_{P}(t)|^2 \vd t
\le \xi^\tr [P(0) - \iRc(0)] \xi
<\infty.
\]
Letting $m\to \infty$ and from Fatou's lemma, we know that 
$u_{P}\in\ipr^2 (\R^n)$,
thus $x_{P}\in \cpr^2(\R^n)$.
From a known result \cite[Theorem~3.2]{hu2003indefinite}, there is at most one solution of SRE~$(\Cw,\Sw,\Tw)$ in the set $\svbset^{\bd}_{\ge\iRc}$.
Therefore we conclude Theorem~\ref{th3-01}(i).

\emph{Step 4.}
Let $P\in\svbset^{\bd}_{\ge\iRc}$ be the solution to SRE~$(\Cw,\Sw,\Tw)$.
Now we send $m\to\infty$ in \eqref{4-60} and get
\begin{equation}\label{4-61}
\xi^\tr P(0) \xi =
\E \int_0^{T}\!\! \big( u_{P}^\tr \Cw u_{P} + x_{P}^\tr \Sw x_{P}  \big)(t)\vd t
+ \E\big[ x_{P}(T)^\tr P(T) x_{P}(T) \big].
\end{equation}
On the other hand, by Lemma~\ref{lem4-10}, we know that
\[
\xi^\tr P(0) \xi \le J(u;\xi),\quad
\forall\,
\xi\in\R^n,\,
u\in\ipr^2(\R^k).
\]
This along with \eqref{4-61} yields that 
\[
\xi^\tr P(0) \xi = J(u_P;\xi) = \inf_{u\in\ipr^2(\R^k)} J(u;\xi) = V(\xi),
\]
and $u_P = \Ctr(P,\df{P}) x_{P}$ is an optimal feedback control for SLQ~$(\Cw,\Sw,\Tw)$.
Finally, by the stochastic maximum principle again, any solution $(x^*,u^*)$ to SLQ~$(\Cw,\Sw,\Tw)$ coincides with the solution of FBSDE~\eqref{4-09} with $s=0$,
thus the uniqueness of the latter implies the uniqueness of the former.
Therefore, the proof of Theorem~\ref{th3-01} is complete.

\section*{Acknowledgments}
The author would like to thank the
associate editor and referees for their helpful comments and suggestions.

\bibliographystyle{siam}
\bibliography{RicEq}

\end{document}